\documentclass[12pt,a4paper]{amsart}

\usepackage[abbrev]{amsrefs}
\usepackage{amscd}
\usepackage{amssymb}

\theoremstyle{plain}
  \newtheorem{thm}{Theorem}

  \newtheorem{cor}[thm]{Corollary}
  \newtheorem{prop}[thm]{Proposition}

\theoremstyle{definition}

\theoremstyle{remark}
  \newtheorem{rem}[thm]{Remark}

\numberwithin{equation}{section}

\newcommand{\field}[1]{\mathbb{#1}}

\newcommand{\R}{\field{R}}

\begin{document}

\title[Estimate of isodiametric constant]
{Estimate of isodiametric constant\\
for closed surfaces}

\thanks{The author is partially supported by a Grant-in-Aid
for Scientific Research from the Japan Society for the Promotion of Science.}

\begin{abstract}
  We give an explicit estimate of the area of a closed surface by the diameter and
  a lower bound of curvature.
  This is better than Calabi-Cao's estimate \cite{CC:clgeod}
  for a nonnegatively curved two-sphere.
\end{abstract}

\author{Takashi Shioya}

\address{Mathematical Institute, Tohoku University, Sendai 980-8578,
  JAPAN}

\date{\today}




\maketitle

\section{Introduction}
\label{sec:intro}

We are interested in the relation between the area $V$ and the diameter $D$
of a closed two-dimensional Riemannian manifold $M$ with
a lower bound of curvature $\kappa$.
It is a famous conjecture due to A. D. Alexandrov \cite{A:conv}
that, if $\kappa = 0$, then
\[
\frac{V}{D^2} \le \frac{\pi}{2}
\]
and the equality holds only if $M$ is the double of a flat disk.
This conjecture seems to be very hard to solve.
Since Alexandrov raised the conjecture in 1955,
there were several attempts to approach it
(see \cite{CC:clgeod, FK:isodiam, KM:alexconj, Sk:isodiam, Sy:diam-area, Z:conj}),
but it is still open.
A trivial upper bound of $V/D^2$ is $\pi$, which follows from the Bishop
volume comparison theorem.
For a (not necessarily zero) lower bound $\kappa$ of curvature, we have
$V \le v_\kappa(D)$, where $v_\kappa(r)$ denotes the area of an $r$-ball
in a complete simply connected two-dimensional space form
with constant curvature $\kappa$, i.e.,
\[
v_\kappa(r) =
\begin{cases}
  \pi r^2 & \text{if $\kappa = 0$}, \\
  2\pi\kappa^{-1}(1-\cos\sqrt{\kappa}\,r) & \text{if $\kappa > 0$}, \\
  2\pi\kappa^{-1}(1-\cosh\sqrt{-\kappa}\,r) & \text{if $\kappa < 0$}.
\end{cases}
\]
Let $R$ be the \emph{radius of $M$}, i.e.,
\[
R := \inf_{p\in M} \sup_{q \in M} d(p,q).
\]
Note that $D/2 \le R \le D$.
The Bishop volume comparison theorem also implies
\begin{equation} \label{eq:VR}
V \le v_\kappa(R).
\end{equation}
This gives a good estimate of $V/D^2$
only in the case where $R$ is far less than $D$.
In this paper, we give a new estimate of $V/D^2$
which is effective for $R$ close to $D$ (see Proposition \ref{prop:VDR}).
Combining this with \eqref{eq:VR},
we obtain the following main theorem.

\begin{thm} \label{thm:VD}
  Let $M$ be a two-dimensional closed Riemannian manifold with
  area $V$, diameter $D$, Euler number $\chi$,
  and with a lower Gaussian curvature bound $\kappa$,
  and let $k := \kappa D^2$.
  We define
  \begin{align*}
    w(k) &:= \frac{v_k(1)}{\pi} =
    \begin{cases}
      1 & \text{if $k = 0$},\\
      2 k^{-1}(1-\cos\sqrt{k}) & \text{if $k > 0$},\\
      2 k^{-1}(1-\cosh\sqrt{-k}) & \text{if $k < 0$},
    \end{cases} \\
    \lambda_\chi(k) &:= 2\chi - k\,w(k), \\
    \alpha_\chi(k) &:=
    \begin{cases}
      w(k) + 360^{-1}k\,\lambda_\chi(k) & \text{if $k \ge 0$},\\
      w(k)^2 & \text{if $k < 0$}.
    \end{cases}
  \end{align*}
  Then we have
  \[
  \frac{V}{D^2} \le v_k\left( \sqrt{
  \frac{-6\,w(k) + 6\sqrt{w(k)^2 + 3^{-1} \lambda_\chi(k)\,\alpha_\chi(k)}}
    {\lambda_\chi(k)}
   } \right),
  \]
  provided $\lambda_\chi(k) > 0$.
\end{thm}

$\lambda_\chi(k)$ is a quantity of how much the surface is far from
the space form.
Note in fact that we always have $\lambda_\chi(k) \ge 0$ and
that $\lambda_\chi(k) = 0$ holds
if and only if $M$ has \emph{constant} nonnegative Gaussian curvature
(see Proposition \ref{prop:lambda}),
for which we know the exact range of $V/D^2$ (see Remark \ref{rem}).
Taking the limit of the estimate in the theorem
as $\lambda_\chi(k) \to 0+$ implies
that $V/D^2 \le v_k(1)$ for $\lambda_\chi(k) = 0$, which is optimal
in the case of constant positive curvature.

The right-hand side of the inequality in the theorem
is continuous in $k$ even at $k = 0$.

Letting $\kappa := 0$ in the theorem yields the following.

\begin{cor} \label{cor:VD}
  Let $M$ be a two-dimensional Riemannian manifold
  with nonnegative Gaussian curvature
  such that $M$ is homeomorphic to either a two-sphere $S^2$
  or a real projective plane $\R P^2$.
  Then we have
  \[
  \frac{V}{D^2} \le \frac{\sqrt{9+6\chi} - 3}{\chi} \pi <
  \begin{cases}
   2.486 & \text{if $M \simeq S^2$},\\
   2.743 & \text{if $M \simeq \R P^2$.}
   \end{cases}
  \]
\end{cor}

In the nonnegative curvature case,
there were a few results for the estimate of $V/D^2$
\cite{CC:clgeod,Sk:isodiam,Sy:diam-area}.
For $M \simeq S^2$,
the best one that I had ever known was $8/\pi \doteqdot 2.548$,
which is due to Calabi and Cao \cite{CC:clgeod}, where their estimate
follows from the first nonzero eigenvalue estimate
by Hirsch \cite{H:isop} and Zhong-Yang \cite{ZY:eigen}.
Our estimate, Corollary \ref{cor:VD}, is better than Calabi-Cao's estimate.
For $M \simeq \R P^2$,
Corollary \ref{cor:VD} is also better than the estimate in
our previous paper \cite{Sy:diam-area}.
The idea of the proof in this paper is different from
any former results.

\section{Estimate by Radius}
\label{sec:radius}

The purpose of this section is to prove the following proposition,
which is the key to the proof of the main theorem.

\begin{prop} \label{prop:VDR}
  With the notations as in Theorem \ref{thm:VD}, we have
  \[
  V \le
  \begin{cases}
  \pi D^2 - \dfrac{2\pi\,\chi\,\tilde v_\kappa(R)}{v_\kappa(D)}
    & \text{if $\kappa \ge 0$}, \\
  \dfrac{\pi D^2 - 2\pi\,\chi\,\tilde v_\kappa(R)/v_\kappa(D)
  - \kappa\,\tilde v_\kappa(D)}
  {1 - \kappa\,\tilde v_\kappa(R)/v_\kappa(D)}
  & \text{if $\kappa < 0$},
  \end{cases}
  \]
  where $R$ denotes the radius of $M$ and
  \begin{align*}
  \tilde v_\kappa(r) &:= \int_0^r \int_0^t v_\kappa(s)\;ds\,dt \\
  &=
  \begin{cases}
    \pi r^4/12 & \text{if $\kappa = 0$,}\\
    \pi \kappa^{-2} (\kappa r^2 + 2\cos(\sqrt{\kappa}\,r) - 2) & \text{if $\kappa > 0$,}\\
    \pi \kappa^{-2} (\kappa r^2 + 2\cosh(\sqrt{-\kappa}\,r) - 2)
    & \text{if $\kappa < 0$.}
  \end{cases}
  \end{align*}
\end{prop}

As a direct consequence of Proposition \ref{prop:VDR},
we have the following.

\begin{cor} \label{cor:VDR}
  Let $M$ be a two-dimensional closed Riemannian manifold with
  nonnegative Gaussian curvature.  Then we have
  \[
  \frac{V}{D^2} \le \pi \left( 1 - \frac{\chi}{6} \rho^4 \right),
  \]
  where $\rho := R/D$.
\end{cor}

\begin{proof}[Proof of Proposition \ref{prop:VDR}]
We set, for $p \in M$,
\[
R_p := \sup_{q \in M} d(p,q),
\]
where $d$ denotes the Riemannian distance function on $M$.
Note that $R_p$ is continuous in $p \in M$.
Let $B(p,r)$ indicate the $r$-ball centered at $p \in M$
and $\partial B(p,r)$ its boundary.
Denote by $\kappa(p,r)$ the sum of the total geodesic curvature
of $\partial B(p,r)$ and of the exterior angles of $B(p,r)$
at cut points in $\partial B(p,r)$ from $p$,
which is well-defined a.e.~$r \in (\,0,R_p\,)$.
By Fiala-Hartman's discussion (see \cite{F:isop, H:geodpara}
and \cite{SST}*{Theorem 4.4.1}),
the length of $\partial B(p,r)$, say $L(p,r)$,
satisfies
\[
\frac{d}{dr} L(p,r) \le \kappa(p,r)
\]
for a.e.~$r \in (\,0,R_p\,)$.
It follows from the Gauss-Bonnet theorem that
\[
  \kappa(p,r) \le 2\pi - \int_{B(p,r)} K(q) \; dq
\]
for a.e.~$r \in (\,0,R_p\,)$,
because the Euler number of a proper subset of $M$ is at most one,
where $K$ denotes the Gaussian curvature function on $M$
and $dq$ the area (volume) measure with respect to the variable $q \in M$.
Note that $\int_{B(p,r)} K(q) \; dq$ is continuous in $(p,r) \in M \times [\,0,+\infty\,)$.
We therefore have
\[
  V = \int_0^{R_p} L(p,t) \; dt
  \le \iint\limits_{0 \le s \le t \le R_p} \left( 2\pi - \int_{B(p,s)} K(q) \; dq \right) \; ds\,dt,
\]
where we note that the right-hand side is continuous in $p \in M$.
Integrating the both sides of the above inequality with respect to $p \in M$,
we have, by $R_p \le D$,
\begin{equation}  \label{eq:V2}
  V^2
  \le \iiint\limits_{\substack{0 \le s \le t \le D\\ p \in M}} 2\pi \; ds\,dt\,dp
  - I
  = \pi D^2 V - I,
\end{equation}
where
\[
I := \iiiint\limits_{\substack{0 \le s \le t \le R_p\\ d(p,q) < s}}  K(q) \; ds\,dt\,dp\,dq.
\]
We are going to estimate $I$ from below.

We assume $\kappa \ge 0$.
Denote by $V(q,s)$ the area of the ball $B(q,s)$.
Since $K \ge 0$, $R_p \ge R$,
and by the Bishop-Gromov volume comparison theorem, we see that
\begin{align*}
I &\ge \iiiint\limits_{\substack{0 \le s \le t \le R\\ d(p,q) < s}}  K(q) \; ds\,dt\,dp\,dq
= \iiint\limits_{\substack{0 \le s \le t \le R\\ q \in M}}  K(q)\,V(q,s) \; ds\,dt\,dq\\
&\ge \frac{V}{v_\kappa(D)} \iiint\limits_{\substack{0 \le s \le t \le R\\ q \in M}}
K(q)\, v_\kappa(s) \; ds\,dt\,dq
\intertext{and by the Gauss-Bonnet theorem,}
&= \frac{2\pi\chi\, V}{v_\kappa(D)} \iint\limits_{0 \le s \le t \le R} v_\kappa(s) \; ds\,dt
= \frac{2\pi\chi\, \tilde v_\kappa(R)}{v_\kappa(D)}\, V.
\end{align*}
Combining this with \eqref{eq:V2} yields
\[
V^2 \le \pi D^2 V - \frac{2\pi\chi\, \tilde v_\kappa(R)}{v_\kappa(D)}\, V,
\]
which implies
\[
V \le \pi D^2 - \frac{2\pi\chi\, \tilde v_\kappa(R)}{v_\kappa(D)}.
\]

We next assume $\kappa < 0$.
Since $K - \kappa \ge 0$ and $R \le R_p \le D$,
\begin{align*}
I &\ge \iiiint\limits_{\substack{0 \le s \le t \le R\\ d(p,q) < s}}  (K(q)-\kappa)
 \; ds\,dt\,dp\,dq
+ \iiiint\limits_{\substack{0 \le s \le t \le D\\ d(p,q) < s}}  \kappa \; ds\,dt\,dp\,dq\\
&= \iiint\limits_{\substack{0 \le s \le t \le R\\ q\in M}}  (K(q)-\kappa)\,V(q,s)
 \; ds\,dt\,dq
 + \iiint\limits_{\substack{0 \le s \le t \le D\\ q\in M}}  \kappa\,V(q,s) \; ds\,dt\,dq\\
 \intertext{and by the Bishop-Gromov and Bishop inequalities,}
&\ge \frac{V}{v_\kappa(D)} \iiint\limits_{\substack{0 \le s \le t \le R\\ q\in M}} 
(K(q)-\kappa)\,v_\kappa(s) \; ds\,dt\,dq
+ \iiint\limits_{\substack{0 \le s \le t \le D\\ q\in M}}  \kappa\,v_\kappa(s)
\; ds\,dt\,dq\\
&= \frac{\tilde v_\kappa(R)}{v_\kappa(D)} \,(2\pi\chi-\kappa \,V) \,V
+ \kappa\, \tilde v_\kappa(D)\, V,
\end{align*}
which together with \eqref{eq:V2} leads us to the proposition.
This completes the proof.
\end{proof}

\section{Estimate by Diameter} \label{sec:diam}

\begin{prop} \label{prop:lambda}
  With the notations as in Theorem \ref{thm:VD},
  we have
  \[
  \lambda_\chi(k) \ge 0
  \]
  and the equality holds if and only if
  $M$ has constant nonnegative Gaussian curvature.
\end{prop}

\begin{proof}
We first assume $k > 0$.  Then, $\chi = 2$ (i.e., $M \simeq S^2$)
or $\chi = 1$ (i.e., $M \simeq \R P^2$).
Note that $\lambda_\chi(k) = 2(\chi - 1 + \cos\sqrt{k})$.

In the case where $\chi = 2$,
we have $\lambda_\chi(k) \ge 0$, and the equality holds if and only if $k = \pi^2$,
which is equivalent to $M$ being a two-sphere with positive constant curvature
by the Toponogov maximal diameter theorem \cite{CE}*{Theorem 6.5}.

In the case where $\chi = 1$, we have
$\lambda_\chi(k) \ge 0$ (resp.~$= 0$) if and only if $k \le \pi^2/4$ (resp.~$= \pi^2/4$).
The theorem in this case follows from
the Grove-Shiohama diameter sphere theorem \cite{GS:diamsphere}
and the Gromoll-Grove diameter rigidity theorem \cite{GG:diamrigidity}.
  
We next assume $k \le 0$.
By the Gauss-Bonnet theorem and the Bishop volume comparison theorem,
\[
\pi\lambda_\chi(k) = \int_M K(p)\;dp - \kappa\, v_\kappa(D)
\ge \kappa\, (V - v_\kappa(D)) \ge 0.
\]
If $\lambda_\chi(k) = 0$ and $k < 0$, then $K \equiv \kappa$ and $V = v_\kappa(D)$,
so that
the cut-locus at any point $p$ in $M$ is contained in $\partial B(p,D)$,
which is a contradiction (cf.~\cite{GP:nearbdy}).
Therefore, if $\lambda_\chi(k) = 0$, then $k = 0$, which implies
\[
\int_M K(p)\;dp = 2\pi\chi = \pi\lambda_\chi(k) = 0
\]
and so $M$ has a flat metric.
This completes the proof.
\end{proof}

\begin{rem} \label{rem}
  If a closed two-manifold has constant nonnegative Gaussian curvature, then
  it is isometric to either a two-sphere with a metric of constant positive curvature,
  a real projective plane with a metric of constant positive curvature,
  a two-torus with a flat metric, or a Klein bottle with a flat metric.
  It is easy to see that $V/D^2 \le 2$
  for a closed two-manifold with a flat metric.
  For a two-sphere with a metric of constant curvature,
  $V/D^2 = 4/\pi$.
  For a real projective plane with a metric of constant curvature,
  $V/D^2 = 8/\pi$.
\end{rem}

\begin{proof}[Proof of Theorem \ref{thm:VD}]
Assume $\lambda_\chi(k) > 0$.
Scaling the metric of $M$, we may assume that $D = 1$.
Then, $k = \kappa$ and $1/2 \le R \le 1$.

We first prove the theorem in the case where $k \ge 0$.
Note that $\chi \ge 0$ in this case.
Combining Proposition \ref{prop:VDR} with \eqref{eq:VR} implies
\[
\frac{V}{D^2} \le \sup_{1/2 \le r \le 1}
\min\{\; \pi - 2\pi\chi\,\tilde v_k(r)/v_k(1), v_k(r) \;\}.
\]
By considering the monotonicity of $\pi - 2\pi\chi\,\tilde v_k(r)/v_k(1)$ and $v_k(r)$
in $r$, 
the supremum above is attained at a unique solution to the equation
\begin{equation} \label{eq:r}
\pi - 2\pi\chi\,\tilde v_k(r)/v_k(1) = v_k(r).
\end{equation}
We will estimate the solution $r$ from above.
For simplicity we set
$f(r) := \tilde v_k(r)/\pi$.
Since $w(k) = v_k(1)/\pi$ and $v_k(r)/\pi = r^2 - k\,f(r)$, the equation \eqref{eq:r}
implies
\[
\lambda_\chi(k) f(r) = w(k) - w(k) \,r^2.
\]
It follows that
\[
f(r) \ge \frac{r^4}{12} - \frac{k\,r^6}{360} \ge \frac{r^4}{12} - \frac{k}{360},
\]
so that
\[
\lambda_\chi(k) \left( \frac{r^4}{12} - \frac{k}{360} \right) \le w(k) - w(k) \,r^2,
\]
which implies
\begin{equation} \label{eq:quadeq}
  \frac{1}{12}\lambda_\chi(k) \,r^4 + w(k)\,r^2 - \alpha_\chi(k) \le 0.
\end{equation}
We therefore have
\[
r^2 \le \frac{-6\,w(k) + 6\sqrt{w(k)^2 + 3^{-1} \lambda_\chi(k)\,\alpha_\chi(k)}}
  {\lambda_\chi(k)},
\]
which together with $V \le v_k(r)$ proves the theorem in this case.

We next prove the theorem in the case where $k < 0$.
Since the derivative
\[
\frac{d}{d\tilde v}
\left( \frac{\pi - 2\pi\chi\,\tilde v/v_k(1) - k \,\tilde v_k(1)}{1-k\,\tilde v/v_k(1)} \right)
= - \frac{\lambda_\chi(k)}{(1-k\, \tilde v/v_k(1))^2 \,w(k)}.
\]
is negative, the function
\[
\frac{\pi - 2\pi\chi\,\tilde v_k(r)/v_k(1) - k \,\tilde v_k(1)}{1-k\,\tilde v_k(r)/v_k(1)}
\]
is monotone decreasing in $r$.
In the same way as before, we have $V \le v_k(r)$
for a unique solution $r$ to the equation
\begin{equation} \label{eq:r2}
  \frac{\pi - 2\pi\chi\,\tilde v_k(r)/v_k(1) - k \,\tilde v_k(1)}{1-k\,\tilde v_k(r)/v_k(1)}
  = v_k(r).
\end{equation}
For the solution $r$, we have
\[
\pi - 2\pi\chi\,\tilde v_k(r)/v_k(1) - k \,\tilde v_k(1) \ge v_k(r)
\]
and then
\[
\lambda_\chi(k) f(r) + w(k)\,r^2 - w(k)^2 \le 0,
\]
which together with $f(r) \ge r^4/12$ yields \eqref{eq:quadeq}.
This completes the proof of Theorem \ref{thm:VD}.
\end{proof}

\begin{rem}
  In the case where $k < 0$,
  we have, by \eqref{eq:r2},
  \[
  k^2 f(r)^2 + (\lambda_\chi(k)-k\,r^2)f(r) + w(k)\, r^2 - w(k)^2 = 0,
  \]
  which together with $f(r) \ge r^4/12$ implies a quartic inequality for $r^2$,
  leading to a better and much more complex form of the estimate of $V/D^2$.
\end{rem}


\begin{bibdiv}
\begin{biblist}

\bib{A:conv}{book}{
   author={Alexandrow, A. D.},
   title={Die innere Geometrie der konvexen Fl\"achen},
   language={German},
   publisher={Akademie-Verlag, Berlin},
   date={1955},
   pages={xvii+522},
}

\bib{CC:clgeod}{article}{
   author={Calabi, Eugenio},
   author={Cao, Jian Guo},
   title={Simple closed geodesics on convex surfaces},
   journal={J. Differential Geom.},
   volume={36},
   date={1992},
   number={3},
   pages={517--549},
   issn={0022-040X},
}

\bib{CE}{book}{
   author={Cheeger, Jeff},
   author={Ebin, David G.},
   title={Comparison theorems in Riemannian geometry},
   note={Revised reprint of the 1975 original},
   publisher={AMS Chelsea Publishing, Providence, RI},
   date={2008},
   pages={x+168},
   isbn={978-0-8218-4417-5},
}

\bib{F:isop}{article}{
   author={Fiala, F.},
   title={Le probl\`eme des isop\'erim\`etres sur les surfaces ouvertes \`a
   courbure positive},
   language={French},
   journal={Comment. Math. Helv.},
   volume={13},
   date={1941},
   pages={293--346},
   issn={0010-2571},
}

\bib{FK:isodiam}{article}{
  author={Freitas, Pedro},
  author={Krej\v{c}i\v{r}\'{i}k, David},
  title={Alexandrov's isodiametric conjecture and the cut locus of a surface},
  note={preprint, arXiv:1406.0811v1},
}

\bib{GG:diamrigidity}{article}{
   author={Gromoll, Detlef},
   author={Grove, Karsten},
   title={A generalization of Berger's rigidity theorem for positively
   curved manifolds},
   journal={Ann. Sci. \'Ecole Norm. Sup. (4)},
   volume={20},
   date={1987},
   number={2},
   pages={227--239},
   issn={0012-9593},
}

\bib{GP:nearbdy}{article}{
   author={Grove, Karsten},
   author={Petersen, Peter},
   title={Manifolds near the boundary of existence},
   journal={J. Differential Geom.},
   volume={33},
   date={1991},
   number={2},
   pages={379--394},
   issn={0022-040X},
}

\bib{GP:volcomp}{article}{
   author={Grove, Karsten},
   author={Petersen, Peter, V},
   title={Volume comparison \`a\ la Aleksandrov},
   journal={Acta Math.},
   volume={169},
   date={1992},
   number={1-2},
   pages={131--151},
   issn={0001-5962},
   doi={10.1007/BF02392759},
}

\bib{GS:diamsphere}{article}{
   author={Grove, Karsten},
   author={Shiohama, Katsuhiro},
   title={A generalized sphere theorem},
   journal={Ann. of Math. (2)},
   volume={106},
   date={1977},
   number={2},
   pages={201--211},
}

\bib{H:geodpara}{article}{
   author={Hartman, Philip},
   title={Geodesic parallel coordinates in the large},
   journal={Amer. J. Math.},
   volume={86},
   date={1964},
   pages={705--727},
   issn={0002-9327},
}

\bib{H:isop}{article}{
   author={Hersch, Joseph},
   title={Quatre propri\'et\'es isop\'erim\'etriques de membranes
   sph\'eriques homog\`enes},
   language={French},
   journal={C. R. Acad. Sci. Paris S\'er. A-B},
   volume={270},
   date={1970},
   pages={A1645--A1648},
}

\bib{KM:alexconj}{article}{
  author={Krysiak, James},
  author={McCoy, Zachary},
  title={Alexandrov's conjecture: On the intrinsic diameter and surface area of convex surfaces},
  note={preprint},
}

\bib{Sk:isodiam}{article}{
   author={Sakai, Takashi},
   title={On the isodiametric inequality for the $2$-sphere},
   conference={
      title={Geometry of manifolds},
      address={Matsumoto},
      date={1988},
   },
   book={
      series={Perspect. Math.},
      volume={8},
      publisher={Academic Press, Boston, MA},
   },
   date={1989},
   pages={303--315},
}

\bib{SST}{book}{
   author={Shiohama, Katsuhiro},
   author={Shioya, Takashi},
   author={Tanaka, Minoru},
   title={The geometry of total curvature on complete open surfaces},
   series={Cambridge Tracts in Mathematics},
   volume={159},
   publisher={Cambridge University Press, Cambridge},
   date={2003},
   pages={x+284},
   isbn={0-521-45054-3},
   doi={10.1017/CBO9780511543159},
}

\bib{Sy:diam-area}{article}{
   author={Shioya, Takashi},
   title={Diameter and area estimates for $S^2$ and ${\bf P}^2$ with
   non-negatively curved metrics},
   conference={
      title={Progress in differential geometry},
   },
   book={
      series={Adv. Stud. Pure Math.},
      volume={22},
      publisher={Math. Soc. Japan, Tokyo},
   },
   date={1993},
   pages={309--319},
}

\bib{Z:conj}{article}{
   author={Zalgaller, V. A.},
   title={A conjecture on convex polyhedra},
   language={Russian, with Russian summary},
   journal={Sibirsk. Mat. Zh.},
   volume={50},
   date={2009},
   number={5},
   pages={1070--1082},
   issn={0037-4474},
   translation={
      journal={Sib. Math. J.},
      volume={50},
      date={2009},
      number={5},
      pages={846--855},
      issn={0037-4466},
   },
   doi={10.1007/s11202-009-0095-3},
}

\bib{ZY:eigen}{article}{
   author={Zhong, Jia Qing},
   author={Yang, Hong Cang},
   title={On the estimate of the first eigenvalue of a compact Riemannian
   manifold},
   journal={Sci. Sinica Ser. A},
   volume={27},
   date={1984},
   number={12},
   pages={1265--1273},
   issn={0253-5831},
}

\end{biblist}
\end{bibdiv}
\end{document}